\documentclass[11pt]{article}
\usepackage[a4paper,margin=29mm]{geometry}
\usepackage{amsmath,amssymb,amsthm,mathtools}
\usepackage[hidelinks]{hyperref}
\usepackage{enumitem}
\setlist[enumerate]{leftmargin=*,itemsep=2pt,topsep=4pt}
\newtheorem{theorem}{Theorem}[section]
\newtheorem{lemma}[theorem]{Lemma}

\newtheorem{corollary}[theorem]{Corollary}
\theoremstyle{definition}\newtheorem{example}[theorem]{Example}

\numberwithin{equation}{section}
\newcommand{\R}{\mathbb R}
\newcommand{\Div}{\operatorname{div}}
\newcommand{\tr}{\operatorname{tr}}
\newcommand{\Id}{\mathrm{Id}}
\newcommand{\grad}{\operatorname{grad}}
\newcommand{\vol}{\,\mathrm{d}v}
\newcommand{\norm}[1]{\lVert #1\rVert}
\newcommand{\ip}[2]{\langle #1,#2\rangle}
\newcommand{\doi}[1]{\href{https://doi.org/#1}{doi:\,#1}}
\title{Composition laws for homogeneous energies\\and symphonic maps}
\author{Xiang-Zhi Cao \thanks{aaa7756kijlp@163.com}\\[3pt]\small School of Information Engineering, Nanjing Xiaozhuang University\\[-1pt]\small Nanjing 211171, China}
\date{}
\begin{document}
\maketitle
\begin{abstract}
For a spectral energy homogeneous of degree $q>1$ in the differential, we characterize the smooth maps which pull back stationary maps into every Riemannian target. If the scalar density is nonzero, these maps are precisely the $q$-harmonic morphisms. The weak statement requires only $C^1$ regularity in the differential variable. It also covers densities which vanish on the identity of the intermediate tangent space. For the energies of the stress tensor $S$ and its trace-free analogue $T$, the source dimension enters the flux. In dimensions at least two we obtain a complete dichotomy: in equal dimensions the preserving maps are the $2p$-harmonic morphisms, and in unequal dimensions only locally constant maps preserve all targets. A hypersurface of dimension $n$ in a target of dimension $n+1$ detects the obstruction when the source dimension exceeds the intermediate dimension $n$.
\end{abstract}

\section{Introduction}
A real-valued test map has at most one nonzero singular value. A density depending on all the singular values of a differential can therefore agree with the ordinary $q$-energy on every scalar map while giving a different Euler--Lagrange system for maps into a surface or a higher-dimensional target. For example, the Euclidean harmonic map $(x,y)\mapsto(x^2-y^2,2xy,x)$ fails to be stationary for the pullback-metric density $\frac12\norm{f^*h_{\rm eucl}}$ of degree two (Example~\ref{ex:different}). The question is whether scalar tests still determine which inner maps $u:M\to N$ preserve these different vector-valued equations. A positive answer gives a way to lift solutions on $N$ to $M$ without solving another Euler--Lagrange system.

Fuglede and Ishihara characterized harmonic morphisms by harmonicity and horizontal weak conformality \cite{fug,ish}. For the usual $q$-energy, Loubeau proved the corresponding result for $q>2$ and derived the composition formula for arbitrary outer targets at the end of the proof of Proposition~3 in \cite{lou}. Burel and Loubeau treated $1<q<2$ \cite{bur}. These results determine what the scalar tests detect. The pullback-metric energy of Nakauchi and Takenaka \cite{nt} poses a further question: its flux involves the one-form $\sigma_f=df\circ(df^*df)$, whose singular directions carry different weights. A formula for the ordinary $q$-tension alone does not identify the divergence of this flux.

Theorem~\ref{th:main} gives a single test for families of such systems on arbitrary Riemannian manifolds. It rests on two facts. A coisometry preserves $LL^*$ under right composition, and a horizontally weakly conformal $q$-harmonic map pulls back continuous divergence-free fluxes by a weighted differential. The latter assertion is a distributional one, including where either differential vanishes. In particular, a product projection $N\times K\to N$ lifts every smooth stationary solution for these energies to a solution on the product. The same construction can fail for the $S$- and $T$-energies. Theorem~\ref{th:dim} identifies exactly when their dependence on the source dimension obstructs composition. Here scalar tests give the same class of inner maps as before, but the normal equation of a hypersurface detects the dimension change. In the $S$ case we construct the requisite outer stationary map by choosing a second metric on a neighborhood in $N$.

The symphonic morphism problem was considered in an earlier version of this work \cite{cao}. Theorem~2 of that version inferred total geodesicity from scalar tests. Example~\ref{ex:log} gives a counterexample and identifies the missing implication. The scalar characterization is classical. Our contribution is the extension to homogeneous spectral fluxes, the distinction between scalar and inner nondegeneracy, and the dimension obstruction for the two tensor energies. The two-exponent and equidimensional geometric classifications used below are due to Ou and Wei \cite{ow}.

\paragraph{Notation and local variations.}
We write $u:(M,g)\to(N,h)$ and $f:(N,h)\to(P,\widetilde h)$, with orthonormal frames $\{\epsilon_\alpha\}$ on $M$ and $\{e_i\}$ on $N$. For $X\in TM$ the symphonic one-form is
\[
 \sigma_u(X)=\sum_\alpha h\bigl(du(X),du(\epsilon_\alpha)\bigr)\,du(\epsilon_\alpha).
\]
The pullback metric is $(u^*h)(X,Y)=h(du(X),du(Y))$. Raising one source index gives $(u^*h)^\sharp=du^*du$, where $du^*$ denotes the metric adjoint of $du$. Thus $\sigma_u=du\circ(u^*h)^\sharp$ and the $p$-symphonic equation is
\[
 \Div_g\bigl(\norm{u^*h}_g^{p-2}\sigma_u\bigr)=0.
\]
At $p=2$ this is $\Div_g\sigma_u=0$. All tensor norms are Hilbert--Schmidt norms. We use $\nabla$ for the Levi--Civita connection, $\widetilde\nabla$ for induced bundle connections, and $\tau(u)=\tr_g\widetilde\nabla du$.

Metrics and maps are smooth, dimensions are positive, and all assertions about solutions are local. We vary the map with fixed source metric and compact support in $\Omega\Subset M$. For a continuous bundle-valued one-form $B$, $\Div_g B=0$ means $\int\ip{B}{\widetilde\nabla W}\vol_g=0$ for every smooth compactly supported test section $W$. A $q$-harmonic morphism pulls back the continuous representative of each real-valued local weak solution $f\in W^{1,q}_{\rm loc}$ of $\Delta_q f=\Div(|df|^{q-2}df)=0$ to a local weak solution. Appendix~\ref{ap:jet} shows that smooth scalar tests suffice for a smooth inner map. Our vector-valued stationary maps are smooth throughout. Weighted products at zeros always mean their stated continuous extensions.

The dilation of a horizontally weakly conformal map $u$ is denoted by $\lambda$. With $u_\alpha^i=h(du(\epsilon_\alpha),e_i)$, horizontal weak conformality (HWC) means
\[
 \sum_\alpha u_\alpha^i u_\alpha^j=\lambda^2\delta_{ij},
 \qquad\text{equivalently}\qquad du\,du^*=\lambda^2\Id_{TN}.
\]
In particular, $|du|^2=n\lambda^2$ for $n=\dim N$, and $du=0$ at $\lambda=0$.

\section{A pullback theorem for spectral energies}
Let $q>1$. For each target dimension $r$, let $\Phi_r$ be real-valued on the cone of positive semidefinite symmetric $r\times r$ matrices, orthogonally invariant and homogeneous of degree $\frac{q}{2}$, with $\Phi_r(0)=0$. In each source dimension $a$ under consideration require
\begin{equation}\label{eq:F}
 F_{a,r}(L)=\Phi_r(LL^*)\quad\text{to be }C^1\text{ on }\operatorname{Hom}(\R^a,\R^r)\setminus\{0\}.
\end{equation}
In particular the requirement applies at nonzero matrices of deficient rank. Write $F=F_{a,r}$ when its dimensions are clear. Denote its Hilbert--Schmidt gradient by $J_\Phi(L)$, and write $J_\Phi(v)=J_\Phi(dv)$ for the induced one-form along a map. When $\Phi_r$ has a differentiable matrix extension at $H_v=dv\,dv^*$, this gives $J_\Phi(v)=2D\Phi_r(H_v)\circ dv$. The definition by differentiating $F$ also works at deficient rank under \eqref{eq:F}. For the symphonic density $\Phi_r(H)=(2p)^{-1}(\tr H^2)^{\frac{p}{2}}$ one has, in the notation above,
\[
 q=2p,\qquad J_\Phi(u)=\norm{u^*h}_g^{p-2}\sigma_u .
\]
On the unit sphere both $F$ and $J_\Phi$ are bounded. Homogeneity gives $|F(L)|\le C|L|^q$ and $|J_\Phi(L)|\le C|L|^{q-1}$. Thus $F$ is $C^1$ at zero with $J_\Phi(0)=0$. Orthogonal invariance makes these objects intrinsic. Differentiating in parallel orthonormal target frames along a variation gives
\[
 E_\Phi(v;\Omega)=\int_\Omega\Phi_r(H_v)\vol_g,\qquad
 \left.\frac d{dt}\right|_0 E_\Phi(v_t;\Omega)
 =\int_\Omega\ip{J_\Phi(v)}{\widetilde\nabla W}\vol_g .
\]
Here $g$ denotes the source metric of $v$. Stationarity means $\Div_g J_\Phi(v)=0$.

\begin{theorem}[Scalar tests and arbitrary targets]\label{th:main}
Assume \eqref{eq:F} and $\Phi_1(1)\ne0$. For a smooth $u:M^m\to N^n$ the following are equivalent.
\begin{enumerate}
\item $u$ pulls back every smooth local scalar $\Phi$-stationary map to a $\Phi$-stationary map.
\item $u$ is a $q$-harmonic morphism.
\item $u$ is HWC and $\Div A=0$ weakly, where $A=\lambda^{q-2}du$ on $\{\lambda>0\}$ and $A=0$ elsewhere. With the HWC condition retained, $\Div A=0$ can equivalently be written $\Div_g(|du|^{q-2}du)=0$, since $A=n^{\frac{2-q}{2}}|du|^{q-2}du$.
\item For every open $V\subset N$, every Riemannian manifold $P$, and every smooth $\Phi$-stationary $f:V\to P$, the composite on $u^{-1}(V)$ is $\Phi$-stationary.
\end{enumerate}
For HWC $u$ one always has $J_\Phi(u)=\kappa_n A$, where $\kappa_n=\frac{q\Phi_n(I_n)}{n}$. If $\kappa_n\ne0$, the third condition is equivalent to HWC together with $\Phi$-stationarity of $u$. For any HWC $u$ and smooth $f$, wherever $J_\Phi(f)$ and $A$ are $C^1$ in neighborhoods of the respective points,
\begin{equation}\label{eq:class}
 \Div_g J_\Phi(f\circ u)=\lambda^q (\Div_h J_\Phi(f))\circ u+(J_\Phi(f)\circ u)(\Div_g A).
\end{equation}
In particular \eqref{eq:class} holds everywhere if $q>2$ and the functions in \eqref{eq:F} are $C^2$ off zero. No scalar or inner nondegeneracy is needed for the flux identity or for the implication from condition 3 to condition 4.
\end{theorem}

\begin{lemma}[Factorization]\label{le:factor}
If $u$ is HWC, then $J_\Phi(f\circ u)=(J_\Phi(f)\circ u)A$.
\end{lemma}
\begin{proof}
At a regular point write $du=\lambda C$, $CC^*=I_n$, and set $L=df$. For any $X$ of the same size as $L$, $(L+tX)C$ has the same target Gram matrix as $L+tX$. Differentiation gives $J_\Phi(LC)C^*=J_\Phi(L)$. Let $P=C^*C$. Reflection $R=2P-I_m$ fixes $LC$ on the right. Right orthogonal invariance gives $J_\Phi(LC)R=J_\Phi(LC)$, so $J_\Phi(LC)(I_m-P)=0$. Hence
\[
 J_\Phi(LC)=J_\Phi(LC)C^*C=J_\Phi(L)C.
\]
Positive homogeneity of $J_\Phi$ of degree $q-1$ proves the assertion. Both sides vanish if $du=0$ or $df=0$. This proof differentiates in $L$ even if $\Phi_r$ has no differentiable matrix extension at $LL^*$.
\end{proof}

\begin{lemma}[Weak pullback]\label{le:weak}
Suppose $u$ is HWC and $\Div A=0$ weakly. If $f:N\to P$ is smooth and $B\in C^0(T^*N\otimes f^*TP)$ satisfies $\Div_h B=0$ weakly, then $\Div_g((B\circ u)A)=0$ weakly.
\end{lemma}
\begin{proof}
The bound $|A|=\sqrt n\lambda^{q-1}$ gives continuity. Work on $U'\Subset U\subset N$ in coordinates with a smooth frame of $f^*TP$. Put $C^i=\sqrt{\det h}\,h^{ij}B_j$, and denote the connection matrices by $\omega_i$. The equation is $\partial_i C^i+\omega_i C^i=0$ in distributions. On $U'$ mollify $C^i$ by $\rho_\epsilon$ using values in $U$. The divergence error in these coordinates is
\[
 \sum_i\int [\omega_i(y)-\omega_i(z)]C^i(z)\rho_\epsilon(y-z)\,dz,
\]
which tends uniformly to zero by uniform continuity on a compact neighborhood. Undoing the metric density gives smooth $B_\epsilon\to B$ with $\Div B_\epsilon\to0$ uniformly on $U'$. For a test section $W$ supported compactly in $u^{-1}(U')$, test $\Div A=0$ against $(B_\epsilon\circ u)^*W$. The product rule yields
\begin{align*}
 -\int\ip{(B_\epsilon\circ u)A}{\widetilde\nabla W}\vol_g
 &=\int\Big\langle\sum_\alpha(\widetilde\nabla_{du(\epsilon_\alpha)}B_\epsilon)(A(\epsilon_\alpha)),W\Big\rangle\vol_g \\
 &=\int \lambda^q\ip{(\Div_h B_\epsilon)\circ u}{W}\vol_g .
\end{align*}
The last contraction uses $du\,du^*=\lambda^2 I$ on the regular set and both integrands vanish at $du=0$. If $K=\operatorname{supp}W$, the error in the left-hand side is bounded by
$\sup_{u(K)}|B_\epsilon-B|\int_K |A|\,|\widetilde\nabla W|\vol_g$.
On the right one uses $\sup_{u(K)}|\Div_h B_\epsilon|\int_K\lambda^q|W|\vol_g$.
Both bounds tend to zero. A partition of unity on the source localizes arbitrary test sections.
\end{proof}

\begin{proof}[Proof of Theorem~\ref{th:main}]
For a scalar differential, $F(df)=\Phi_1(1)|df|^q$, so the scalar flux is $q\Phi_1(1)|df|^{q-2}df$. Condition 1 is thus preservation of smooth scalar $q$-harmonic functions. Appendix~\ref{ap:jet} proves that these tests imply 3 and that 3 preserves every local scalar weak solution, giving 2. Condition 2 implies 1 by restriction. Lemmas~\ref{le:factor} and \ref{le:weak} imply 4 from 3 with $B=J_\Phi(f)$, and $P=\R$ gives 1 from 4.

Conjugation invariance at $I_n$ gives $J_\Phi(I_n)=cI_n$. Differentiating $F(tI_n)$ at $t=1$ gives $nc=q\Phi_n(I_n)$. Applying Lemma~\ref{le:factor} with the identity outer map gives $J_\Phi(u)=cA$. Finally, the covariant product rule for $(J_\Phi(f)\circ u)A$ contracts $\widetilde\nabla J_\Phi(f)$ with $du\,du^*$, giving \eqref{eq:class}. For $C^2$ off-zero densities, $|DJ_\Phi(L)|\le C|L|^{q-2}$ away from zero. If $q>2$, $J_\Phi$ and $A=n^{\frac{2-q}{2}}|du|^{q-2}du$ are $C^1$ across zero and their derivatives there vanish.
\end{proof}

\paragraph{Scope of the hypotheses.}
A zero scalar coefficient makes every smooth scalar function stationary. For $q=4$ the choice $\Phi_r=e_2$ (second elementary symmetric function, with $e_2=0$ when $r=1$) gives this degeneracy even though $\Phi_2(I_2)>0$; see \cite{rw} for the corresponding higher-power energies. For a planar differential $L$ the density is $F(L)=(\det L)^2$ and its flux is $2\det(L)\operatorname{cof}L$. Take $u(s,t)=(s,\tfrac12t^2)$ on $t>0$. It passes all scalar tests. The outer identity of the Euclidean plane has constant flux, while the composite $u$ has flux $\operatorname{diag}(2t^2,2t)$ and divergence $(0,2)$. Thus the scalar tests can miss failure of arbitrary-target preservation. Conversely, choose $\Phi_1(t)=t^{\frac{q}{2}}$ and $\Phi_2\equiv0$, with ordinary densities in other dimensions. On the plane, $u(x)=\frac{x}{|x|^2}$ is HWC and $\Phi$-stationary. At $q=4$ it fails the scalar test $f(y)=y_1$ by \eqref{eq:invert}. Thus the inner-map characterization can fail when $\kappa_n=0$. At $q=1$ the nonzero scalar density is a multiple of $|t|$ and has no derivative at zero. This is the endpoint of the single-valued $C^1$ first-variation setting here.

The hypothesis on $F$ is strictly weaker than a $C^2$ extension hypothesis on $\Phi$. For example, take $\Phi_r(H)=(\tr H^2)^{\frac{1}{2}}\left|\frac{\tr(H^2)}{(\tr H)^2}-\frac{1}{2}\right|^{\frac{3}{2}}$ for $H\ne0$, with value zero at zero. It has degree one in $H$. In dimension three the ratio crosses $\frac{1}{2}$ transversally along $H=\operatorname{diag}(t,1,1)$ at $t=4$, so even on positive definite matrices the density need not be $C^2$. The scalar coefficient is positive.

\section{Symphonic and trace-power energies}
\paragraph{The symphonic formula.}
The pullback metric in the energy of Nakauchi and Takenaka \cite{nt} also defines, for $p>\frac{1}{2}$, the family
\[
 E_{2,p}(u;\Omega)=\frac1{2p}\int_\Omega\norm{u^*h}_g^p\vol_g .
\]
An inner map $u$ is called a $p$-symphonic morphism if it pulls back every smooth local $p$-symphonic function to a $p$-symphonic function. The flux of $E_{2,p}$ has degree $2p-1$ in $du$ and is extended by zero at $du=0$. For a real-valued $f$ one has $\norm{f^*dt^2}_h=|df|^2$ and $\sigma_f=|df|^2df$. Its equation therefore reduces to $\Delta_{2p}f=0$.

For a horizontally weakly conformal $u$, evaluating the definition of $\sigma$ in the two orthonormal frames gives
\[
 \norm{u^*h}_g=\sqrt n\,\lambda^2,\qquad \sigma_u=\lambda^2du,\qquad
 \sigma_{f\circ u}=\lambda^2(\sigma_f\circ u)\circ du .
\]
In particular, $\norm{u^*h}_g^{p-2}\sigma_u
=n^{\frac{p-2}{2}}\lambda^{2p-2}du$ with the continuous convention at zero. For $p>1$ the pointwise formula of Theorem~\ref{th:main} reads
\begin{align*}
 &\Div_g\!\left(\norm{(f\circ u)^*\widetilde h}_g^{p-2}\sigma_{f\circ u}\right)\\
 &\quad=\lambda^{2p}
       \left[\Div_h\!\left(\norm{f^*\widetilde h}_h^{p-2}\sigma_f\right)\right]\circ u\\
 &\qquad\quad+n^{\frac{2-p}{2}}
       \left(\norm{f^*\widetilde h}_h^{p-2}\sigma_f\right)_{u(x)}
       \!\left[\bigl(\Div_g(\norm{u^*h}_g^{p-2}\sigma_u)\bigr)(x)\right].
\end{align*}
The last line evaluates a one-form at $u(x)$ on a vector in $T_{u(x)}N$. For $\frac{1}{2}<p\le1$, preservation still holds in the weak sense of Corollary~\ref{co:trace}.

\paragraph{Trace powers.}
In this section $m\ge2$ denotes the trace order, independently of the source dimension. Put $q=mp$ and $p>\frac{1}{m}$. Let $X\in TM$ and write $u_\alpha=du(\epsilon_\alpha)$. The cyclic contraction and its associated one-form are
\begin{align*}
 \norm{d_m(u)}^2
 &=\sum_{\alpha_1,\ldots,\alpha_m}\prod_{j=1}^{m}h(u_{\alpha_j},u_{\alpha_{j+1}}),
       \qquad \alpha_{m+1}=\alpha_1,\\
 \sigma_{u,m}(X)
 &=\sum_{\alpha_1,\ldots,\alpha_{m-1}}
 h(du(X),u_{\alpha_1})
 \left(\prod_{j=1}^{m-2}h(u_{\alpha_j},u_{\alpha_{j+1}})\right)u_{\alpha_{m-1}} .
\end{align*}
An empty product is one. Thus $\sigma_{u,2}=\sigma_u$ and $\norm{d_2(u)}=\norm{u^*h}_g$. In terms of the source endomorphism $(u^*h)^\sharp=du^*du$ the same contractions give
\begin{equation}\label{eq:trace}
 \norm{d_m(u)}^2=\tr\big((du^*du)^m\big),\qquad
 \sigma_{u,m}=du\circ(du^*du)^{m-1}.
\end{equation}
The trace-power energy and its Euler--Lagrange equation are
\[
 E_{m,p}(u;\Omega)=\frac1{mp}\int_\Omega\norm{d_m(u)}^p\vol_g,
 \qquad
 \Div_g\!\left(\norm{d_m(u)}^{p-2}\sigma_{u,m}\right)=0.
\]
Write $E_m=E_{m,2}$. We use the subscript order $\sigma_{u,m}$ from Lemma~3 of \cite{cao}; the introduction of that preprint also wrote $\sigma_{m,u}$ for the same form. For $m=3$ and $p=2$ this gives the $\Phi_{(3)}$-energy \cite{phi3}.

\begin{corollary}\label{co:trace}
For $mp>1$, scalar preservation for $E_{m,p}$, arbitrary-target preservation for $E_{m,p}$, and the $mp$-harmonic morphism property are equivalent. They are equivalent to HWC together with weak $E_{m,p}$-stationarity of the inner map. For HWC $u$ and $mp>2$ one has at $x$
\begin{equation}\label{eq:tracepoint}
\begin{aligned}
 &\Div_g\!\left(\norm{d_m(f\circ u)}^{p-2}\sigma_{f\circ u,m}\right)\\
 &\quad=\lambda^{mp}
 \left[\Div_h\!\left(\norm{d_m(f)}^{p-2}\sigma_{f,m}\right)\right]\circ u\\
 &\qquad\quad+n^{\frac{2-p}{2}}
 \left(\norm{d_m(f)}^{p-2}\sigma_{f,m}\right)_{u(x)}
 \!\left[\bigl(\Div_g(\norm{d_m(u)}^{p-2}\sigma_{u,m})\bigr)(x)\right].
\end{aligned}
\end{equation}
For $m=2$, weak preservation thus holds throughout $p>\frac{1}{2}$.
\end{corollary}
\begin{proof}
For $\Phi_r(H)=(mp)^{-1}(\tr H^m)^{\frac{p}{2}}$ and $H_u=du\,du^*$, the nonzero eigenvalues of $H_u$ and $du^*du$ agree. Differentiation at $du\ne0$ gives
\[
 J_\Phi(u)=(\tr H_u^m)^{\frac{p}{2}-1}H_u^{m-1}du
           =\norm{d_m(u)}^{p-2}\sigma_{u,m}.
\]
The last equality uses $H_u^{m-1}du=du(du^*du)^{m-1}$. The whole one-form extends continuously by zero at $du=0$. For scalar $f$ it reduces to $|df|^{mp-2}df$. At the identity $\kappa_n=n^{\frac{p-2}{2}}$. Theorem~\ref{th:main} gives the assertions.
\end{proof}
The calculation also applies to $q^{-1}\sum_j c_j(\tr H^j)^{\frac{q}{2j}}$ for finite mixtures. Real coefficients are allowed if $\sum_j c_j\ne0$ for the equivalence of the two preservation properties with the $q$-morphism property. For the inner stationarity test one also needs $\sum_j c_j n^{\frac{q}{2j}-1}\ne0$. Taking all $c_j\ge0$ with some $c_j>0$ ensures both conditions.

If $f^*\widetilde h=\lambda_f^2h$ with $\lambda_f>0$, then
\[
 \norm{d_m(f)}^{p-2}\sigma_{f,m}
 =n^{\frac{p-2}{2}}\lambda_f^{mp-2}df .
\]
The last line of \eqref{eq:tracepoint} becomes $(\lambda_f\circ u)^{mp-2}df_{u(x)}$ applied to the inner divergence. The symphonic exponent is $2p-2$. At $p=2$, the outer factor for trace order $m$ is $\lambda_f^{2m-2}df$. This corrects $\lambda_f^m df$ in \cite[Lemma 3]{cao}. For instance, in dimension one take $m=3$, $u(x)=x^2$ on $x>0$ and $f(y)=2y$. At $x=1$ the derivatives of the fluxes $(2x)^5$ and $(4x)^5$ are $160$ and $5120$. Their ratio is $32=2^4\cdot2$, whereas $2^3\cdot2=16$. For a general smooth outer map $f$, retain the assumption that $u$ is HWC. If both $f$ and $f\circ u$ are stationary and $mp>2$, \eqref{eq:tracepoint} gives
\[
 \bigl(\Div_g(\norm{d_m(u)}^{p-2}\sigma_{u,m})\bigr)(x)\in\ker df_{u(x)} .
\]
Indeed, by singular value decomposition the kernel of the extended outer flux equals $\ker df$, including at $df=0$. In particular, for HWC $u$, a stationary immersion $f$ detects inner stationarity whenever $f\circ u$ is stationary. Injectivity of $df_{u(x)}$ on a dense subset of the source also suffices by continuity of the inner divergence.

The HWC hypothesis cannot be omitted here. For the ordinary symphonic equation ($m=p=2$ in this section), let $w(x)=\log|x|$ on $\R^4\setminus\{0\}$, and on its product with $\R$ set
\[
 u(x,t)=(w(x)+t,t),\qquad f(a,b)=(a-b,b).
\]
The affine diffeomorphism $f$ has constant flux. For $v=f\circ u=(w,t)$ the component fluxes of $\sigma_v=dv\,dv^*dv$ are $|dw|^2dw$ and $dt$, whose divergences vanish because $\Delta_4 w=0$. For $u$, the second row of $du\,du^*$ is $(1,1)$, so the second component of $\sigma_u$ is $dw+2dt$. Its divergence is $\Delta w=2|x|^{-2}\ne0$; the off-diagonal entry $1$ in $du\,du^*$ also exhibits the failure of HWC.

\begin{example}\label{ex:different}
Take $F_D(L)=\frac{|L|^2}{2}$, $F_{\rm sym}(L)=\frac{(\tr((L^*L)^2))^{\frac{1}{2}}}{2}$. Both have $q=2$. For $f(x,y)=(x^2-y^2,2xy,x)$ put $\rho=x^2+y^2$. Its components are harmonic and $df^*df=\operatorname{diag}(4\rho+1,4\rho)$. For $F_{\rm sym}$ the third component of $\norm{f^*g_{\rm eucl}}^{-1}\sigma_f=\frac{df(df^*df)}{(32\rho^2+8\rho+1)^{\frac{1}{2}}}$ is $\left(\frac{4\rho+1}{(32\rho^2+8\rho+1)^{\frac{1}{2}}},0\right)$. Its divergence is $-\frac{32\rho x}{(32\rho^2+8\rho+1)^{\frac{3}{2}}}$ and is nonzero at $(1,0)$.
\end{example}

\paragraph{The pointwise endpoint.}
At $q=2$ a positively one-homogeneous flux is differentiable at zero only if it is linear: put $tL$ into the definition of the derivative and let $t\downarrow0$. For the trace-power energy with $mp=2$, its flux sends $E_{11},E_{22}$ to themselves and $I_2$ to $2^{\frac{1}{m}-1}I_2$, precluding differentiability there. The matrix test establishes loss of $C^1$ regularity as a function of the differential. For $1<q<2$, the scalar example $f(t)=\frac{t^2}{2}$ already has a flux $|t|^{q-2}t$ without a finite derivative at zero. The Dirichlet density has a linear flux and satisfies the classical formula at $q=2$.

\begin{corollary}[Two exponents]\label{co:two}
A smooth map belongs to two trace-power morphism classes of distinct effective degrees $q_1,q_2>1$ exactly when it is HWC and
\[
 \tau(u)=0,\qquad du(\grad \lambda^2)=0.
\]
It then belongs to all the classes in Corollary~\ref{co:trace}.
\end{corollary}
\begin{proof}
On $\{\lambda>0\}$, $\Div A=\lambda^{q-2}\{\tau(u)+(q-2)du(\grad\log\lambda)\}$. Subtraction yields both equations there. Since $\lambda^2=\frac{|du|^2}{n}$ is smooth, the equations extend to the closure of the regular set and hold on the interior of its complement. Conversely the divergence of $(\lambda^2+\epsilon)^{\frac{q-2}{2}}du$ vanishes by the product rule. For $1<q<2$ its norm on $\{\lambda\le\delta\}$ is at most $\sqrt n\,\delta^{q-1}$, and convergence is uniform away from zero on compact sets. For $q\ge2$, continuity directly gives local uniform convergence. Passage to test sections gives $\Div A=0$ everywhere. This is the two-exponent principle of \cite[Corollary 2.3]{ow} in the present convention.
\end{proof}

\begin{example}\label{ex:log}
For $u(x)=\log |x|$ on $\R^4\setminus\{0\}$, $du\ne0$ and $\Div(|du|^2\grad u)=\Div(\frac{x}{|x|^4})=0$. Every noncritical scalar map is horizontally conformal. Hence $u$ is a symphonic morphism for $p=2$, although its radial Hessian is $-|x|^{-2}$. More explicitly, $\Delta u=2|x|^{-2}$ and $\nabla^2u(\partial_r,\partial_r)=-|x|^{-2}$, so the two terms of the normalized $4$-Laplace equation cancel. The last step in the proof of Theorem~2 of \cite{cao} treated $(u_\alpha^k u_\beta^k)_{\alpha,\beta}$ for fixed $k$ as positive definite. This matrix has rank at most one. Vanishing of its contraction with a symmetric Hessian does not imply that the Hessian vanishes.
\end{example}

\section{The two dimension-dependent energies}
Here $m=\dim M$ and $n=\dim N$. For $u:(M^m,g)\to(N^n,h)$ put
\[
 S_u=u^*h-\tfrac12|du|^2g,\qquad
 T_u=u^*h-\tfrac1m|du|^2g,\qquad
 E_{Q,p}(u;\Omega)=\frac{1}{2p}\int_\Omega\norm{Q_u}_g^p\vol_g .
\]
For the underlying tensors and their energies see \cite{nakT,hanT,fengS,fhw}. The exponent of homogeneity in $du$ is $2p$. In terms of the symphonic one-form $\sigma_u$, the two first variations involve
\begin{equation}\label{eq:Q}
 \sigma_{S,u}=\sigma_u+\frac{m-4}{4}|du|^2du,\qquad
 \sigma_{T,u}=\sigma_u-\frac1m|du|^2du .
\end{equation}
Indeed, in source dimension $d$ let $c_d^S=\frac{d-4}{4}$ and $c_d^T=-\frac{1}{d}$. For a map $v:(L^d,g)\to(P,\eta)$ one has
\[
 \norm{Q_v}_g^2=\norm{v^*\eta}_g^2+c_d^Q|dv|^4,\qquad
 \sigma_{Q,v}=dv\circ\big(dv^*dv+c_d^Q|dv|^2\Id\big).
\]
Differentiating the squared norm gives $4\ip{\sigma_{Q,v}}{\widetilde\nabla W}$. Hence the stationary equation is $\Div_g(\norm{Q_v}_g^{p-2}\sigma_{Q,v})=0$. In the proof below $J_Q(v)$ denotes this weighted one-form, with its continuous zero value whenever that extension is justified.

If $Q=T$ or $Q=S$ with $d=2$, the range $p>1$ ensures a continuous flux, since $\sigma_{T,v}=dv\circ T_v^\sharp$ and $\sigma_{S,v}=dv\circ(S_v^\sharp-\tfrac12\tr_g S_v\,I_d)$. For $d\ne2$, the trace of $S_v$ gives
\begin{equation}\label{eq:coerc}
 \norm{S_v}\ge \left|1-\frac{d}{2}\right|\,\frac{|dv|^2}{\sqrt{d}} .
\end{equation}
The $S$-density is then smooth off $dv=0$, and $C^1$ across $dv=0$ for $2p>1$. At the endpoint $2p=1$ its restriction to a nonzero scalar matrix line is a positive multiple of $|t|$. In the other cases, vary $I_d$ in a nonzero symmetric trace-free direction $H$ ($d\ge2$). The tensor is $2tH+O(t^2)$ at $I_d+tH$, so its $p$-power has no ordinary derivative at $t=0$ for $0<p\le1$. These tests give the exact uniform ranges for first differentiability at all differentials with arbitrary targets.

\begin{theorem}[Dimension dichotomy]\label{th:dim}
Let $m,n\ge2$, let $u:M^m\to N^n$ be smooth, and consider preservation of every smooth local $E_{Q,p}$-stationary map into every Riemannian target. For $Q=T$ assume $p>1$. For $Q=S$ assume $p>\frac{1}{2}$ when both $m,n\ne2$, and $p>1$ when either dimension is two.
\begin{enumerate}
\item If $m=n=d$, universal preservation is equivalent to being a $2p$-harmonic morphism. On each connected component, these are exactly the HWC maps if $d=2p$ and exactly the constants and local homotheties if $d\ne2p$.
\item If $m\ne n$, universal preservation holds precisely for locally constant maps. Scalar preservation alone still characterizes the $2p$-harmonic morphisms.
\end{enumerate}
\end{theorem}

\begin{proof}[Scalar tests and equal dimensions]
For a scalar differential the squared norm is $b_d|df|^4$ and $\sigma_{Q,f}=b_d|df|^2df$, where $b_d=\frac{d}{4}$ for $S$ and $b_d=\frac{d-1}{d}$ for $T$. Thus $J_Q(f)=b_d^{\frac{p}{2}}|df|^{2p-2}df$. The source and intermediate constants can differ and both are positive. Universal preservation entails scalar preservation, hence HWC and $\Div A=0$ for $q=2p$.

When $m=n$ and $du=\lambda C$, $C$ is orthogonal on the regular set. Thus $Q_{f\circ u}=u^*Q_f$, $\norm{Q_{f\circ u}}=\lambda^2\norm{Q_f}$ and $\sigma_{Q,f\circ u}=\lambda^2\sigma_{Q,f}du$. Consequently, including the zero values,
\[
 \norm{Q_{f\circ u}}_g^{p-2}\sigma_{Q,f\circ u}
 =\left[\bigl(\norm{Q_f}_h^{p-2}\sigma_{Q,f}\bigr)\circ u\right]\circ(\lambda^{2p-2}du).
\] Lemma~\ref{le:weak} gives sufficiency. For $p\ge2$ the classical product formula \eqref{eq:class} holds with $J_Q$; for $S$, $d\ne2$ it also holds for $p>1$ by \eqref{eq:coerc}. On the regular set the conformal connection formula gives
\[
 \tau(u)=(2-d)du(\grad\log\lambda),\qquad
 \Div A=(2p-d)\lambda^{2p-2}du(\grad\log\lambda).
\]
If $2p=d$, the latter vanishes there and, for the admissible critical cases here, $d>2$ and $A$ is $C^1$ with derivative zero at $du=0$. If $2p\ne d$, $du$ is invertible at regular points, so $\lambda$ is a positive constant on each regular component. Such a component is closed in its source component: continuity keeps $\lambda>0$ at each limit point and a connected regular neighborhood there meets the component. This proves the geometric classification (cf.\ \cite[Theorem 3.1]{ow}). A local homothety has parallel $du$ and constant dilation, so satisfies the weak equation for all admissible $p$.
\end{proof}

\subsection{Tests in unequal dimensions}
We give the remaining proof of Theorem~\ref{th:dim}. If $m<n$, HWC forces $du=0$. Suppose $m>n$ and $du_x\ne0$. Set $y=u(x)$ and write $du_x=\lambda C$, $CC^*=I_n$. Consider an outer immersion $i:(V,h)\to(P,\eta)$ with induced metric $k=i^*\eta$ and chosen unit normal $\nu$. Write $H=h^{-1}k=di^*di$ and let $D$ be its second fundamental form in the direction $\nu$, viewed also as an endomorphism using $h$. For a tensor weight $c$ define
\[
 W_c(H)=H+c\,\tr(H)I_n .
\]
At a point where $Q_{i\circ u}\ne0$, the normal component of its Euler--Lagrange field is
\begin{equation}\label{eq:normal}
 \ip{\Div J_Q(i\circ u)}{\nu}
   =\norm{Q_{i\circ u}}^{p-2}\lambda^4\tr(W_{c_m^Q}(H)D).
\end{equation}
For this calculation put $G_{i\circ u}=d(i\circ u)^*d(i\circ u)$. Then $G_{i\circ u}=\lambda^2 C^*HC$ and
$C(G_{i\circ u}+c_m^Q\tr G_{i\circ u}\,I_m)C^*=\lambda^2 W_{c_m^Q}(H)$.
The two remaining factors $du$ in the normal Hessian of $i\circ u$ contribute $\lambda^2$. Derivatives of the coefficient and $di(\widetilde\nabla du)$ are tangent. Thus \eqref{eq:normal} holds without differentiating the dilation or the weight in the normal direction. Where $Q_i\ne0$, the same computation for $i$ gives $\norm{Q_i}^{p-2}\tr(W_{c_n^Q}(H)D)$. On an open set where $Q_i\equiv0$, the full extended flux vanishes identically. In particular we can test the identity slice even when $1<p<2$. The tangential equation is exactly the equation for the same map into the slice with metric $k$.

A prescribed smooth symmetric $D$ on $V$ can be realized as this second fundamental form. Equip $(-\epsilon,\epsilon)\times V$ with $dt^2+k_t$, $k_t=k-2tD^\flat$, taking $D^\flat$ to be the chosen covariant form on $(V,h)$ and restricting $V$ and $\epsilon$ to maintain positivity. The slice at zero, with normal $\partial_t$, has second fundamental form $-\frac12\partial_tk_t|_0=D^\flat$.

\paragraph{The trace-free tensor and the surface case.}
Take the identity slice $k=h$, $D^\flat=h$. Its $T$ tensor vanishes identically, so $i$ is stationary. For $i\circ u$, the horizontal eigenvalues of $G_{i\circ u}$ are $\lambda^2$ and the vertical ones are zero. Hence $T_{i\circ u}\ne0$ and the unweighted normal trace in \eqref{eq:normal} is
\[
 \lambda^4 n(1-\frac{n}{m})>0.
\]
This contradicts universal preservation. For $S$ and $n=2$, one also has $S_i=0$, whereas the trace for $i\circ u$ is $\lambda^4 n(1+c_m^S n)=\lambda^4(m-2)>0$ and $S_{i\circ u}\ne0$. These tests apply to every intermediate metric.

\paragraph{The stress tensor when $m>n\ge3$.}
At $H=I_n$ the two normal weights are proportional. To separate them, take the metric $k$ given by Lemma~\ref{le:chart} below. It makes the identity $v:(V,h)\to(V,k)$ stationary for $E_{S,p}$, while $H=h^{-1}k$ has two different eigenvalues at $y$. Write $W_j=W_{c_j^S}(H)$. By choosing $H$ close to $I_n$ we have $W_n\ne0$ throughout $V$, since $W_{c_n^S}(I_n)=\tfrac14(n-2)^2 I_n$. Prescribe the second fundamental form by
\begin{equation}\label{eq:orth}
 D=W_m-\frac{\ip{W_m}{W_n}}{|W_n|^2}W_n .
\end{equation}
Here $\ip{X}{Y}=\tr(XY)$ for $h$-self-adjoint endomorphisms. The difference $W_m-W_n=\tfrac14(m-n)\tr(H)I_n$ is a nonzero scalar matrix. If $W_m$ and $W_n$ were proportional at $y$, this identity would force $W_n$, hence $H$, to be scalar there. Consequently $D_y\ne0$. Equation~\eqref{eq:orth} gives $\tr(W_n D)=0$ throughout $V$ and $\tr(W_mD)=|D|^2$.

In the metric $dt^2+k-2tD^\flat$, the slice immersion $i$ is stationary: its tangential equation is that of $v$, and its normal equation has the factor $\tr(W_nD)$. For $i\circ u$, \eqref{eq:coerc} gives $\norm{S_{i\circ u}}>0$ near $x$. Its normal component at $x$ in \eqref{eq:normal} is
$\norm{S_{i\circ u}}^{p-2}\lambda^4|D_y|^2>0$. The smooth Euler--Lagrange field is therefore nonzero on an open set, contradicting weak stationarity. Locally constant maps have zero flux. This proves Theorem~\ref{th:dim} subject to the following local construction.

\begin{lemma}[A stationary identity with unequal singular values]\label{le:chart}
Let $n\ge3$, $p>\frac{1}{2}$ and $y\in (N^n,h)$. On a neighborhood $V$ of $y$ there is a smooth Riemannian metric $k$ such that $v=\Id:(V,h)\to(V,k)$ is $E_{S,p}$-stationary and $H=h^{-1}k$ is nonscalar at $y$. On a smaller relatively compact neighborhood $H$ may be taken arbitrarily close to $I$.
\end{lemma}
\begin{proof}
Write $c=\tfrac14(n-4)$ and $b=1+cn=\tfrac14(n-2)^2>0$. For a positive $h$-self-adjoint endomorphism $H$ set
\begin{equation}\label{eq:metricmap}
 R(H)=H^2+c\,\tr(H)H=H W_c(H),\qquad P(H)=\tr R(H).
\end{equation}
For $k(X,Y)=h(HX,Y)$ one has $P(H)=\norm{S_v}^2$. We first determine a sufficient condition for $v$ to be stationary. Write $J=J_S(v)$, $e=\frac1{2p}P(H)^{\frac{p}{2}}$ and, in an $h$-orthonormal frame, set
\[
 \mathcal T_{ij}=e\,h_{ij}-k(J(e_i),dv(e_j)).
\]
At a point where the frame is normal, differentiating $e$ gives
$\partial_j e=\sum_i k(J(e_i),(\widetilde\nabla_jdv)(e_i))$.
The symmetry $(\widetilde\nabla_jdv)(e_i)=(\widetilde\nabla_i dv)(e_j)$ cancels this sum when taking the divergence of $\mathcal T$. Thus
\begin{equation}\label{eq:stress}
 (\Div_h\mathcal T)_j=-k(\Div_h J,dv(e_j)),\qquad
 \mathcal T^\sharp=eI-P(H)^{\frac{p-2}{2}}R(H).
\end{equation}
Indeed, pairing $dv\circ W_c(H)$ with $dv$ contributes the extra factor $H$ in $R(H)$. Since $dv$ is onto, $\Div_h\mathcal T=0$ implies the entire map equation $\Div_h J=0$.

Choose a smooth $n$-harmonic function $\psi$ with $d\psi_y\ne0$. The scalar jet theorem \cite[Theorem 2 and Proposition 2]{lou} gives it by prescribing a unit gradient and zero Hessian in normal coordinates, as in Appendix~\ref{ap:jet}. Shrink $V$ so that $d\psi$ has no zeros. Define $B$ by
\[
 B^\flat=\frac{|d\psi|^n}{n}\,h-|d\psi|^{n-2}d\psi\otimes d\psi .
\]
Its trace is zero. In its divergence, the derivative of $\frac{|d\psi|^n}{n}$ cancels $|d\psi|^{n-2}\nabla^2\psi(\grad\psi,\cdot)$; hence
$\Div_h B^\flat=-(\Delta_n\psi)d\psi=0$.
The eigenvalues of $B_y$ along $\grad\psi$ and transverse to it are respectively
$-\frac{(n-1)|d\psi|^n}{n}$ and $\frac{|d\psi|^n}{n}$.

We now solve $R(H)=bI+\epsilon B$ for $H$ near $I$. We have $R(I)=bI$ and
\[
 DR_I[K]=(2+cn)K+c\,\tr(K)I .
\]
On trace-free matrices this derivative is multiplication by $1+b$; on scalar matrices it is multiplication by $2b$. Both numbers are positive. In an $h$-orthonormal frame the finite-dimensional inverse function theorem gives a smooth inverse around $bI$ with positive definite values near $I$. Orthogonal conjugation commutes with $R$ and hence with this local inverse by uniqueness. On a relatively compact $V$, choose $\epsilon\ne0$ sufficiently small to obtain a smooth $H$ throughout $V$. It approaches $I$ as $\epsilon\to0$. At $y$, $bI+\epsilon B_y$ is nonscalar, so $H_y$ is nonscalar. Finally, $P(H)=nb$ is constant and $\Div_h(R(H)^\flat)=0$. By \eqref{eq:stress} the associated identity map is stationary.
\end{proof}

\paragraph{How many target dimensions must be tested?}
In Theorem~\ref{th:dim} it suffices to test all Riemannian targets of dimension $n+1$. Indeed, composing a scalar map with the linear isometric inclusion $\R\hookrightarrow\R^{n+1}$ leaves its pullback metric unchanged and embeds its flux in the same fixed line. These tests already enforce scalar preservation. For $m>n$ the distinguishing immersion constructed above also has target dimension $n+1$. Its ambient metric is allowed to vary with the test.

\paragraph{One-dimensional exceptions.}
If $n=1$, then $df^*df$ has rank at most one for every outer map $f$. For $S$ on a one-dimensional intermediate manifold and a source of any dimension, an HWC inner map yields
$J_S(f\circ u)= \left(\frac{b_m}{b_1}\right)^{\frac{p}{2}}(J_S(f)\circ u)A$,
where $b_j=\frac{j}{4}$. This follows from the rank-one formula
$J_S(L)=b_j^{\frac{p}{2}}|L|^{2p-2}L$ and $|d(f\circ u)|=|df|\lambda$.
The product has the usual continuous zero extension. Thus the universal $S$-morphisms to a curve are precisely the $2p$-harmonic morphisms for $p>\frac{1}{2}$ if $m\ne2$, and for $p>1$ if $m=2$ in the full $C^1$ energy setting. In fact the rank-one identity gives the same scalar test for $p>\frac{1}{2}$ when $m=2$: every composite has rank at most one, so at a nonzero differential its $S$ tensor is nonzero even in dimension two. To justify first variation along arbitrary smooth variations in this last case, write the differential in parallel frames as $L_t=L_0+O(t)$ uniformly on a compact set. The nonzero rank-one directions on the unit sphere are compact and have a neighborhood on which $F=\frac{\norm{S}^p}{2p}$ is smooth with $|DF(L)|\le C|L|^{2p-1}$. If $|L_t-L_0|\le\delta|L_0|$ for sufficiently small $\delta$, use this bound along the segment. Otherwise $|L_t|+|L_0|=O(|t|)$ and $|F(L_t)-F(L_0)|=O(|t|^{2p})$. These estimates justify differentiation under the integral for $2p>1$. For $m=1<n$, scalar tests force $S$-morphisms to be locally constant. For $T$, all maps from a one-dimensional source have zero energy and every inner map with $m=1$ preserves stationary maps. If $m>1=n$, all outer scalar functions of an arc-length coordinate $s$ are $T$-stationary, while their composites must be $2p$-harmonic. Testing with $s$ and $e^s$ forces $d(s\circ u)=0$: on the regular set, writing $w=s\circ u$,
$\Delta_q e^w=e^{(q-1)w}(\Delta_q w+(q-1)|dw|^q)$. This explains the lower dimension bound in the joint dichotomy.

\paragraph{Two concrete tests.}
Euclidean inversion $u(x)=\frac{x}{|x|^2}$ in dimension $d\ge2$ satisfies $T_u=0$ and also $S_u=0$ for $d=2$. The linear scalar test $f(y)=y_1$ has composite $v=\frac{x_1}{|x|^2}$, and for $q=2p$ direct differentiation gives
\begin{equation}\label{eq:invert}
 |dv|=|x|^{-2},\qquad \Delta_q v=2(q-d)x_1|x|^{-2q}.
\end{equation}
Thus stationarity of the inner map may lose the required scalar information. For a second test, take $S^n\times K^{m-n}\to S^n$ with the product metric and compose with the unit-sphere inclusion $i:S^n\hookrightarrow\R^{n+1}$. Writing $F$ for the composite and $a=n(1-\frac{n}{m})>0$, one has $T_i=0$, $\norm{T_F}^2=a$, $\sigma_{T,F}=(1-\frac{n}{m})dF$ and $\tau(F)=-nF$. Hence $\Div J_T(F)=-a^{\frac{p}{2}}F$. For $n=2$, $S_i=0$ and the same computation gives $\Div J_S(F)=-(m-2)^{\frac{p}{2}}F$. These are direct instances of the unequal-dimension obstruction.

\appendix
\section{Why smooth scalar tests suffice}\label{ap:jet}
Here is the characterization used in Theorem~\ref{th:main}. At $y\in N$, choose a unit vector $v$ and symmetric $H$ satisfying
\begin{equation}\label{eq:jet}
 \tr H+(q-2)H(v,v)=0.
\end{equation}
In normal coordinates the normalized $q$-Laplace equation has principal matrix $I+(q-2)v\otimes v$ at this jet. Its eigenvalues are $1$ in transverse directions and $q-1$ longitudinally. The smooth elliptic jet-realization theorem stated in \cite[Theorem 2]{lou} therefore supplies a smooth local solution with prescribed gradient $v$ and Hessian $H$. Its hypotheses apply near the nonzero jet for every $q>1$. Shrinking the open set keeps the gradient nonzero, where normalization is legitimate.

Suppose all smooth local scalar solutions pull back. At a point with $U=du_x\ne0$, let $G=UU^*$ and choose a unit eigenvector $v$ of positive eigenvalue $\kappa$. Compare the solutions with jets $(v,H)$ and $(v,0)$. Their composite gradients agree and equal $w=U^*v\ne0$. In the Hessian chain rule, the terms containing $\widetilde\nabla du$ cancel. The remaining trace is $\tr(GH)$, and the normalized directional term is $\frac{H(Uw,Uw)}{|w|^2}=\kappa H(v,v)$. Here $Uw=\kappa v$ and $|w|^2=\kappa$. The two composite gradients agree at $x$, so the normalized equations have the same directional coefficient at that point. Subtracting there gives
\[
 \tr(GH)+(q-2)\kappa H(v,v)=0
 \quad\text{for all }H\text{ satisfying \eqref{eq:jet}}.
\]
The functional in \eqref{eq:jet} is nonzero, since its value on $v\otimes v$ is $q-1$. Kernel containment makes the two functionals proportional. Identifying symmetric matrices with their trace dual gives $G+(q-2)\kappa v\otimes v=c(I+(q-2)v\otimes v)$. Evaluation on $v$ yields $c=\kappa$, hence $G=\kappa I$. This also rules out intermediate positive ranks.

Around each $y$ take $n$ solutions $f_j$ with independent nonzero gradients there, using $H=0$. Their fluxes $C_j=|df_j|^{q-2}df_j$ form a smooth coframe after restriction. By HWC, $|d(f_j\circ u)|^{q-2}d(f_j\circ u)=(C_j\circ u)A$ everywhere. For smooth $C_j$ the distributional product rule gives
\[
 0=(C_j\circ u)\Div A+\lambda^q(\Div C_j)\circ u=(C_j\circ u)\Div A.
\]
Here $A$ is continuous and the differentiated coefficient is smooth. Multiplication by the smooth inverse coframe is valid even on the critical set. Thus $\Div A=0$. Conversely, let $u$ be HWC with $\Div A=0$ and let $f\in W^{1,q}_{\rm loc}(V)$ be any real-valued weak $q$-harmonic function. Scalar interior regularity on a smooth Riemannian manifold gives a $C^{1,\alpha}_{\rm loc}$ representative for some $\alpha>0$ \cite[Proposition 2.3]{ls}. In coordinates the divergence coefficient is $\sqrt{\det h}\,|\xi|_{h^{-1}}^{q-2}h^{-1}\xi$, which satisfies the growth and monotonicity hypotheses on relatively compact coordinate balls. In dimension one the same regularity follows directly in arc length from constancy of $|f'|^{q-2}f'$. Thus $C=|df|^{q-2}df$ is a continuous, weakly divergence-free ordinary one-form. The scalar case of Lemma~\ref{le:weak} applies to $C$ on the trivial real line bundle (one may take the auxiliary bundle map there to be constant). The classical first derivative chain rule for this representative and HWC give
\[
 |d(f\circ u)|^{q-2}d(f\circ u)=(C\circ u)A
\]
with the continuous zero extension. Since $f\circ u$ is $C^1$ and the right side has weak divergence zero, $f\circ u$ is a local weak $q$-harmonic function.

\section*{Use of generative AI}
 The  first version \cite{cao} was done without any AI tools.   ChatGPT (OpenAI) assisted with English drafting, \LaTeX{} preparation, literature searches, and the development and checking of mathematical arguments. The author is responsible for the content of the manuscript.

\end{document}